\documentclass[10pt]{amsart}
\usepackage{amsmath,amssymb,latexsym,soul,cite,amsthm,color,enumitem,graphicx,tikz, mathtools,microtype}
\usepackage[colorlinks=true,urlcolor=burgundy,citecolor=burgundy,linkcolor=burgundy,linktocpage,pdfpagelabels,bookmarksnumbered,bookmarksopen]{hyperref}
\definecolor{burgundy}{rgb}{0.5, 0.0, 0.13}
\usepackage[english]{babel}
\usepackage[left=3.2cm,right=3.2cm,top=2.9cm,bottom=2.9cm]{geometry}

\numberwithin{equation}{section}

\newtheorem{theorem}{Theorem}[section]
\theoremstyle{plain}

\theoremstyle{plain}
\newtheorem{proposition}[theorem]{Proposition}
\theoremstyle{plain}

\newtheorem{definition}[theorem]{Definition}
\theoremstyle{definition}
\newtheorem{remark}[theorem]{Remark}
\newtheorem{example}[theorem]{Example}

\newcommand{\N}{{\mathbb N}}

\newcommand{\R}{{\mathbb R}}
\newcommand{\eps}{\varepsilon}
\newcommand{\s}{H^s_0(\Omega)}
\newcommand{\h}{{\bf H}_1}
\newcommand{\hh}{{\bf H}_2}
\newcommand{\beq}{\begin{equation}}
\newcommand{\eeq}{\end{equation}}
\renewcommand{\le}{\leqslant}
\renewcommand{\ge}{\geqslant}
\newcommand{\fl}{(-\Delta)^{s\,}}

\newenvironment{enumroman}{\begin{enumerate}

}{\end{enumerate}}

\title[Nonlinear fractional Laplacian equations]{Three nontrivial solutions for nonlinear fractional Laplacian equations}

\author[F.G.\ D\"uzg\"un, A. Iannizzotto]{Fatma Gamze D\"uzg\"un, Antonio Iannizzotto}

\address[F.G.\ D\"uzg\"un]{Department of Mathematics
\newline\indent
Hacettepe University
\newline\indent
06800, Beytepe, Ankara, Turkey}
\email{gamzeduz@hacettepe.edu.tr}

\address[A.\ Iannizzotto]{Department of Mathematics and Computer Science
\newline\indent
University of Cagliari
\newline\indent
Viale L. Merello 92, 09123 Cagliari, Italy}
\email{antonio.iannizzotto@unica.it}


\begin{document}

\begin{abstract}
We study a Dirichlet-type boundary value problem for a pseudodifferential equation driven by the fractional Laplacian, proving the existence of three nonzero solutions. When the reaction term is sublinear at infinity, we apply the second deformation theorem and spectral theory. When the reaction term is superlinear at infinity, we apply the mountain pass theorem and Morse theory.
\end{abstract}

\maketitle

\begin{center}
Version of \today\
\end{center}

\section{Introduction}\label{sec1}

\noindent
The present paper deals with the following Dirichlet-type boundary value problem for a nonlinear equation driven by the fractional Laplacian:
\beq\label{pro}
\begin{cases}
\fl u=f(x,u) & \text{in $\Omega$} \\
u=0 & \text{in $\Omega^c$,}
\end{cases}
\eeq
where $\Omega\subseteq\R^N$ ($N>1$) is a bounded domain with a $C^2$ boundary, $\Omega^c=\R^N\setminus\Omega$, $s\in(0,1)$, and $f:\Omega\times\R\to\R$ is a Carath\'eodory function. The fractional Laplacian operator is defined for any sufficiently smooth function $u:\R^N\to\R$ and all $x\in\R$ by
\beq\label{fl}
\fl u(x)=C_{N,s}\lim_{\eps\to 0^+}\int_{\R^N\setminus B_\eps(x)}\frac{u(x)-u(y)}{|x-y|^{N+2s}}\,dy,
\eeq
where $C_{N,s}>0$ is a suitable normalization constant. Throughout the paper we will always assume $C_{N,s}=1$ (for a precise evaluation of $C_{N,s}$, consistent with alternative definitions of the fractional Laplacian, see \cite[Remark 3.11]{CS1}).
\vskip2pt
\noindent
Fractional operators have gained increasing popularity in recent years. This is due both to the intrinsic mathematical interest of such subject, and to the various applications that they allow. Indeed, nonlocal pseudodifferential operators such as $\fl$ are naturally involved in continuum mechanics, population dynamics, game theory and other phenomena, as the infinitesimal generators of L\'evy-type stochastical processes (see \cite{C1}).
\vskip2pt
\noindent
Roughly speaking, the outstanding feature of operators like $\fl$ is {\em nonlocality}, i.e., the dependence of $\fl u(x)$ on the values of $u(y)$ not only for $y$ conveniently near to $x$, but but for all $y\in\R^N$. While such nonlocality makes our operator particularly suitable to describe phenomena allowing 'jumps', it makes things delicate in dealing with regularity, sign, and other typically {\em local} attributes of solutions. This is one reason why the study of nonlinear equations involving $\fl$ (or closely related operators) started with the case when the domain is $\R^N$, providing existence of solutions, regularity, {\em a priori} bounds and maximum principles (see \cite{CS1,CS2}, and \cite{AP} for some existence results). The natural functional setting for such study is provided by fractional Sobolev spaces (see \cite{DPV}).
\vskip2pt
\noindent
On the other hand, nonlocality obviously produces some difficulties in finding an analogous to Dirichlet-type boundary conditions on bounded domains. The standard formulation of the Dirichlet problem for fractional equations in a bounded domain $\Omega$ was set in the series of papers \cite{SV,SV1,SV2}, simply by requiring that the solution $u$ vanishes a.e.\ outside $\Omega$. Our problem \eqref{pro} follows such standard. While interior regularity of solutions of \eqref{pro} can be handled just as in the unbounded case, boundary regularity and behaviour of solutions (e.g.\ the Hopf property) came forth as a serious difficulty, which was mostly overcome by means of weighted H\"older-type function spaces (see \cite{BCSS,GS1,IMS,RS}).
\vskip2pt
\noindent
Once provided with the appropriate functional formulation, problem \eqref{pro} becomes variational, in the sense that its weak solutions can be detected as critical points of a $C^1$ energy functional $\varphi$, defined on a fractional Sobolev space. So, we can prove existence and multiplicity of such solutions by applying to $\varphi$ several abstract results of critical point theory, such as minimax principles (see \cite{R}) and Morse theory (see \cite{C}). Some results of this type can be found, for instance, in \cite{BMS,CW,F,IS,MR,MP,ZF}.
\vskip2pt
\noindent
In the present paper, we will employ much of the research accomplished so far in order to prove the existence of three nonzero solutions for problem \eqref{pro} (one positive, one negative, and the third with indefinite sign), when $f(x,\cdot)$ has a subcritical growth and satisfies convenient conditions at zero and at infinity. Precisely, we will consider two cases:
\begin{itemize}[leftmargin=1cm]
\item[{\bf (a)}] if $f(x,\cdot)$ is {\em sublinear} at infinity, and at most linear at zero, then we apply the second deformation theorem and some spectral properties of $\fl$ (namely, a characterization of the second eigenvalue which, for the local case, goes back to \cite{CFG});
\item[{\bf (b)}] if $f(x,\cdot)$ is {\em superlinear} at infinity, and satisfies a mild version of the Ambrosetti-Rabinowitz condition, then we apply the mountain pass theorem and the Poincar\'e-Hopf identity based on the computation of critical groups (thus proving a nonlocal analogous of the result of \cite{W}).
\end{itemize}
In both cases, truncations of the energy functional $\varphi$ will be an essential tool, so we will make use of a topological result established in \cite{IMS}, which relates local minimizers of the truncated and uncut functionals, respectively.
\vskip2pt
\noindent
Our work strongly relies on the joint application of mutually independent results, and we decided to privilege simplicity rather than generality. One possible generalization of our results is towards linear, nonlocal operators of the type
\[\mathcal{L}_K u(x)=\lim_{\eps\to 0^+}\int_{\R^N\setminus B_\eps(x)}\frac{u(x)-u(y)}{K(x,y)}\,dy,\]
where $K:\R^N\times\R^N\to\R_+$ is a weight function exhibiting an asymptotic behaviour similar to that of the standard weight $|x-y|^{N+2s}$ (see \cite{SV}). Another possible extension may deal with the fractional $p$-Laplacian, namely the nonlinear, nonlocal operator defined by
\[(-\Delta)_p^s u(x)=2\lim_{\eps\to 0^+}\int_{\R^N\setminus B_\eps(x)}\frac{|u(x)-u(y)|^{p-2}(u(x)-u(y))}{|x-y|^{N+2s}}\,dy,\]
where $p\in(1,\infty)$. Some existence and multiplicity results for fractional $p$-Laplacian problems, obtained through critical point theory and Morse theory, can be found in \cite{ILPS}. Nevertheless, the methods used in the present paper cannot be easily extended to $(-\Delta)_p^s$ due to the lack of a complete boundary regularity theory like that developed in \cite{RS} for $\fl$ (some results in this direction are proved in \cite{IMS1}).
\vskip2pt
\noindent
The paper has the following structure: in Section \ref{sec2} we recall the variational formulation of our problem and some basic properties of solutions, together with some results from critical point theory; in Section \ref{sec3} we prove our multiplicity result for the sublinear case; and in Section \ref{sec4} we deal with the superlinear case.

\section{Preliminary results}\label{sec2}

\noindent
In this section we recall some results that will be used in our arguments.

\subsection{Variational formulation and some properties of problem \eqref{pro}}\label{ss21}

For all measurable function $u:\R^N\to\R$ we set
\[[u]_{s,2}^2=\iint_{\R^N\times\R^N}\frac{(u(x)-u(y))^2}{|x-y|^{N+2s}}\,dx\,dy,\]
then we define the fractional Sobolev space
\[H^s(\R^N)=\{u\in L^2(\R^N):\,[u]_{s,2}<\infty\}\]
(see \cite{DPV}). We restrict ourselves to the subspace
\[H^s_0(\Omega)=\{u\in H^s(\R^N):\,u(x)=0 \ \text{for a.e. $x\in\Omega^c$}\},\]
which is a separable Hilbert space under the norm $\|u\|=[u]_{s,2}$ (see \cite{SV}). We denote by $H^{-s}(\Omega)$ the topological dual of $\s$ and by $\langle\cdot,\cdot\rangle$ the scalar product of $\s$ (or the duality pairing between $H^{-s}(\Omega)$ and $\s$). In this connection we mention the following useful inequality, holding for all $u\in\s$:
\beq\label{np}
\iint_{\R^N\times\R^N}\frac{(u(x)-u(y))(u^-(x)-u^-(y))}{|x-y|^{N+2s}}\,dx\,dy\le -\|u^-\|^2,
\eeq
where $u^-$ stands for the negative part of $u$ (see \cite{IMS}). The critical exponent is defined as $2^*_s=\frac{2N}{N-2s}$, and the embedding $H^s_0(\Omega)\hookrightarrow L^p(\Omega)$ is continuous and compact for all $p\in[1,2^*_s)$ (see \cite[Lemma 8]{DPV}). Moreover, we introduce the positive order cone
\[H^s_0(\Omega)_+=\{u\in H^s_0(\Omega):\,u(x)\ge 0 \ \text{for a.e. $x\in\Omega$}\},\]
which has an empty interior with respect to the $H^s_0(\Omega)$-topology. The space $H^s_0(\Omega)$ provides the natural framework for the study of problem \eqref{pro}:

\begin{definition}\label{ws}
A function $u\in H^s_0(\Omega)$ is a (weak) solution of \eqref{pro} if for all $v\in H^s_0(\Omega)$
\[\iint_{\R^N\times\R^N}\frac{(u(x)-u(y))(v(x)-v(y))}{|x-y|^{N+2s}}\,dx\,dy=\int_\Omega f(x,u)v\,dx.\]
\end{definition}

\noindent
In all the forthcoming results we will assume the following hypothesis on the nonlinearity $f$:
\begin{itemize}[leftmargin=1cm]
\item[${\bf H}_0$] $f:\Omega\times\R\to\R$ is a Carath\'eodory mapping, satisfying
\[|f(x,t)|\le a_0(1+|t|^{p-1}) \ \text{for a.e. $x\in\Omega$ and all $t\in\R$ ($a_0>0$, $p\in(1,2^*_s)$).}\]
\end{itemize}
Under such assumption, we are able to extend to problem \eqref{pro} some basic results holding for elliptic boundary value problems, starting with a simple {\em a priori} bound:

\begin{proposition}\label{apb}
{\rm \cite[Theorem 3.2]{IMS}} Let ${\bf H}_0$ hold. Then there exists a continuous, nondecreasing function $M:\R_+\to\R_+$ s.t.\ for all $u\in H^s_0(\Omega)$ weak solution of \eqref{pro} one has $u\in L^\infty(\Omega)$ and
\[\|u\|_\infty\le M(\|u\|_{2^*_s}).\]
\end{proposition}

\noindent
While solutions of fractional equations exhibit good interior regularity properties, they may have a singular behaviour on the boundary. So, instead of the usual space $C^1(\overline\Omega)$, they are better embedded in the following weighted H\"older-type spaces. Set $\delta(x)={\rm dist}(x,\Omega^c)$ for all $x\in\R^N$ and define
\[C^0_\delta(\overline\Omega)=\Big\{u\in C^0(\overline\Omega):\,\frac{u}{\delta^s}\in C^0(\overline\Omega)\Big\},\]
\[C^\alpha_\delta(\overline\Omega)=\Big\{u\in C^0(\overline\Omega):\,\frac{u}{\delta^s}\in C^\alpha(\overline\Omega)\Big\} \ (\alpha\in(0,1)),\]
endowed with the norms
\[\|u\|_{0,\delta}=\Big\|\frac{u}{\delta^s}\Big\|_\infty, \ \|u\|_{\alpha,\delta}=\|u\|_{0,\delta}+\sup_{x\neq y}\frac{|u(x)/\delta^s(x)-u(y)/\delta^s(y)|}{|x-y|^\alpha},\]
respectively. For all $0\le\alpha<\beta<1$ the embedding $C^\beta_\delta(\overline\Omega)\hookrightarrow C^\alpha_\delta(\overline\Omega)$ is continuous and compact. In this case, the positive cone $C^0_\delta(\overline\Omega)_+$ has a nonempty interior given by
\beq\label{cone}
{\rm int}\,(C^0_\delta(\overline\Omega)_+)=\Big\{u\in C^0_\delta(\overline\Omega):\,\frac{u(x)}{\delta^s(x)}>0 \ \text{for all $x\in\overline\Omega$}\Big\}.
\eeq
From Proposition \ref{apb} and \cite[Theorem 1.2]{RS} we have the following global regularity result:

\begin{proposition}\label{reg}
Let ${\bf H}_0$ hold. Then there exist $\alpha\in(0,\min\{s,1-s\})$ and $C>0$ s.t.\ for all $u\in H^s_0(\Omega)$ weak solution of \eqref{pro} one has $u\in C^\alpha_\delta(\overline\Omega)$ and
\[\|u\|_{\alpha,\delta}\le C(1+\|u\|_{2^*_s}).\]
\end{proposition}

\noindent
We turn now to sign properties of solutions of \eqref{pro}. We begin with a weak maximum principle:

\begin{proposition}\label{wmp}
{\rm \cite[Theorem 2.4]{IMS}} Let ${\bf H}_0$ hold and $f(x,t)\ge 0$ for a.e. $x\in\Omega$ and all $t\in\R$. If $u\in H^s_0(\Omega)$ is a solution of \eqref{pro}, then $u$ is lower semicontinuous and $u(x)\ge 0$ for all $x\in\Omega$.
\end{proposition}

\noindent
Moreover, we have the following fractional Hopf lemma:

\begin{proposition}\label{hl}
{\rm \cite[Lemma 1.2]{GS1}} Let ${\bf H}_0$ hold and $f(x,t)\ge -ct$ for a.e. $x\in\Omega$ and all $t\in\R$ ($c>0$). If $u\in H^s_0(\Omega)_+$ is a solution of \eqref{pro}, $u$ lower semicontinuous, then either $u(x)=0$ for all $x\in\Omega$, or $u\in{\rm int}\,(C^0_\delta(\overline\Omega)_+)$.
\end{proposition}

\begin{remark}
In its original version from \cite{GS1}, the above Hopf lemma requires that $u$ satisfies $\fl u=f(x,u)$ {\em pointwisely} in $\Omega$, while we deal with weak solutions. In fact, any weak solution $u$ of \eqref{pro} has a higher interior regularity than that displayed in Proposition \ref{reg}, as $u\in C^{1,\beta}(\Omega)$ for any $\beta\in(\max\{0,2s-1\},2s)$ (see \cite[Corollary 5.6]{RS}). Hence, also recalling that $u=0$ in $\Omega^c$, one can see that the limit in \eqref{fl} exists in $\R$ and the equation is satisfied pointwisely (see \cite[Proposition 2.12]{IMS1}).
\end{remark}

\noindent
Now we introduce an energy functional for problem \eqref{pro}. Set for all $(x,t)\in\Omega\times\R$
\[F(x,t)=\int_0^t f(x,\tau)\,d\tau,\]
and for all $u\in H^s_0(\Omega)$
\beq\label{phi}
\varphi(u)=\frac{\|u\|^2}{2}-\int_\Omega F(x,u)\,dx.
\eeq
By the continuous embedding $H^s_0(\Omega)\hookrightarrow L^p(\Omega)$ we have $\varphi\in C^1(H^s_0(\Omega))$, and for all $u,v\in H^s_0(\Omega)$
\[\varphi'(u)(v)=\iint_{\R^N\times\R^N}\frac{(u(x)-u(y))(v(x)-v(y))}{|x-y|^{N+2s}}\,dx\,dy-\int_\Omega f(x,u)v\,dx.\]
So, recalling Definition \ref{ws}, $u$ is a solution of \eqref{pro} iff $\varphi'(u)=0$ in $H^{-s}(\Omega)$. Among critical points of $\varphi$, {\em local minimizers} play a preeminent role. We recall, in this connection, a useful topological result relating such minimizers in the $H^s_0(\Omega)$-topology and in $C^0_\delta(\overline\Omega)$-topology, respectively (a fractional version of the classical result of \cite{BN}):

\begin{proposition}\label{lm}
{\rm \cite[Theorem 1.1]{IMS}} Let ${\bf H}_0$ hold, $\varphi$ be defined as above, and $u\in H^s_0(\Omega)$. Then, the following conditions are equivalent:
\begin{enumroman}
\item\label{lm1} there exists $r>0$ s.t. $\varphi(u+v)\ge\varphi(u)$ for all $v\in H^s_0(\Omega)$, $\|v\|\le r$;
\item\label{lm2} there exists $\rho>0$ s.t. $\varphi(u+v)\ge\varphi(u)$ for all $v\in H^s_0(\Omega)\cap C^0_\delta(\overline\Omega)$, $\|v\|_{0,\delta}\le\rho$.
\end{enumroman}
\end{proposition}

\noindent
In the proof of our result we will need some spectral properties of $\fl$. Let us consider the following eigenvalue problem:
\beq\label{ep}
\begin{cases}
\fl u=\lambda u & \text{in $\Omega$} \\
u=0 & \text{in $\Omega^c$.}
\end{cases}
\eeq
Just as in the local case, we say that $\lambda>0$ is an eigenvalue of $\fl$ if problem \eqref{ep} has a nonzero solution $u\in H^s_0(\Omega)$, which is called a $\lambda$-eigenfunction. From the current literature we have rather complete information about the first two eigenvalues of $\fl$:

\begin{proposition}\label{ev}
The spectrum of $\fl$ consists of a nondecressing sequence $0<\lambda^1_s(\Omega)<\lambda^2_s(\Omega)\le \ldots$ of positive numbers, in particular:
\begin{enumroman}
\item\label{ev1}{\rm \cite[Proposition 9]{SV1}} $\lambda^1_s(\Omega)$ is simple and the unique $L^2(\Omega)$-normalized eigenfunction is $\hat u_1\in {\rm int}\,(C^0_\delta(\overline\Omega)_+)$ s.t. $\|\hat u_1\|_2=1$, moreover $\lambda^1_s(\Omega)$ admits the variational characterization
\[\lambda^1_s(\Omega)=\inf_{u\in H^s_0(\Omega)\setminus\{0\}}\frac{\|u\|^2}{\|u\|_2^2};\]
\item\label{ev2}{\rm \cite[Proposition 2.8]{GS}} $\lambda^2_s(\Omega)$ is the smallest eigenvalue in $(\lambda^1_s(\Omega),\infty)$, the $\lambda^2_s(\Omega)$-eigenfunctions are nodal, moreover $\lambda^2_s(\Omega)$ admits the variational characterization
\[\lambda^2_s(\Omega)=\inf_{\gamma\in\Gamma_1}\max_{t\in[0,1]}\|\gamma(t)\|^2,\]
where
\[\Gamma_1=\{\gamma\in C([0,1],H^s_0(\Omega)):\,\gamma(0)=\hat u_1,\,\gamma(1)=-\hat u_1,\,\|\gamma(t)\|_2=1 \ \text{\rm for all $t\in[0,1]$}\}.\]
\end{enumroman}
\end{proposition}

\noindent
Note that \ref{ev2} above is a fractional version of a classical result of \cite{CFG}, and that Proposition \ref{ev} holds as well for $(-\Delta)_p^s$ (see \cite{BP,FP}). For further information about the spectra of $\fl$ and $(-\Delta)_p^s$ see also \cite{IS1,PSY,SV2}.

\subsection{Some recalls of critical point theory}\label{ss22}

Variational methods are based on abstract critical point theory, and the latter includes many results, depicting the rich topology that nonlinear and nonconvex functionals may exhibit. We recall here some well-known results which will be our major tools, mainly following \cite{MMP} (see also \cite{R}).
\vskip2pt
\noindent
Let $(X,\|\cdot\|)$ be a reflexive Banach space, $(X^*,\|\cdot\|_*)$ be its topological dual, and $\varphi\in C^1(X)$ be a functional. By $K(\varphi)$ we denote the set of all critical points of $\varphi$, i.e., those ponts $u\in X$ s.t. $\varphi'(u)=0$ in $X^*$, while for all $c\in\R$ we set
\[K_c(\varphi)=\{u\in K(\varphi):\,\varphi(u)=c\},\]
besides we set
\[\overline\varphi^c=\{u\in X:\,\varphi(u)\le c\}.\]
Most results require the following Cerami compactness condition (a weaker version of the Palais-Smale condition):
\begin{align}\tag{C}\label{c} 
\begin{split}
&\text{Any sequence $(u_n)$ in $X$, s.t. $(\varphi(u_n))$ is bounded in $\R$ and $(1+\|u_n\|)\varphi(u_n)\to 0$}\\
&\text{in $X^*$, admits a (strongly) convergent subsequence.}
\end{split}
\end{align}
We recall a version of the mountain pass theorem (see \cite{AR,PS} for the original result):

\begin{theorem}\label{mpt}
{\rm \cite[Theorem 5.40]{MMP}} Let $\varphi\in C^1(X)$ satisfy \eqref{c}, $u_0,u_1\in X$, $r\in(0,\|u_1-u_0\|)$ be s.t.
\[\max\{\varphi(u_0),\varphi(u_1)\}<\eta_r:=\inf_{\|u-u_0\|=r}\varphi(u),\]
moreover let
\[\Gamma=\{\gamma\in C([0,1],X): \ \gamma(0)=u_0,\,\gamma(1)=u_1\},\]
\[c=\inf_{\gamma\in\Gamma}\max_{t\in[0,1]}\varphi(\gamma(t)).\]
Then, $c\ge\eta_r$ and $K_c(\varphi)\neq\emptyset$.
\end{theorem}

\noindent
We will also use the second deformation theorem:

\begin{theorem}\label{sdt}
{\rm \cite[Theorem 5.34]{MMP}} Let $\varphi\in C^1(X)$ satisfy \eqref{c}, $a<b$ be real numbers s.t. $K_c(\varphi)=\emptyset$ for all $c\in(a,b)$ and $K_a(\varphi)$ is a finite set. Then, there exists a continuous deformation $h:[0,1]\times(\overline\varphi^b\setminus K_b(\varphi))\to(\overline\varphi^b\setminus K_b(\varphi))$ s.t.
\begin{enumroman}
\item\label{sdt1} $h(0,u)=u$, $h(1,u)\in\overline\varphi^a$ for all $u\in(\overline\varphi^b\setminus K_b(\varphi))$;
\item\label{sdt2} $h(t,u)=u$ for all $(t,u)\in[0,1]\times\overline\varphi^a$;
\item\label{sdt3} $t\mapsto\varphi(h(t,u))$ is decreasing in $[0,1]$ for all $u\in(\overline\varphi^b\setminus K_b(\varphi))$.
\end{enumroman}
\end{theorem}

\noindent
In particular, \ref{sdt1} - \ref{sdt2} above mean that $\overline\varphi^a$ is a {\em strong deformation retract} of $\overline\varphi^b$ (see \cite[Definition 5.33 (b)]{MMP}).
\vskip2pt
\noindent
We conclude this section by recalling some basic notions from Morse theory (see \cite{BSW,C} for details). Let $\varphi\in C^1(X)$ satisfy \eqref{c} and $u\in K_c(\varphi)$ ($c\in\R$) be an {\em isolated} critical point of $\varphi$, i.e., there exists a neighborhood $U\subset X$ of $u$ s.t. $K(\varphi)\cap U=\{u\}$. Then, for all integer $k\ge 0$ the $k$-{\em th critical group of $\varphi$ at $u$} is defined as
\beq\label{cg}
C_k(\varphi,u)=H_k(\overline\varphi^c\cap U,\overline\varphi^c\cap U\setminus\{u\}),
\eeq
where $H_k(\cdot,\cdot)$ is the $k$-th (singular) homology group of a topological pair (see \cite[Definition 6.9]{MMP}). All these groups are real linear spaces. Note that, by the excision property of homology groups, \eqref{cg} is invariant with respect to $U$. In particular, if $u\in K(\varphi)$ is a strict local minimizer and an isolated critical point, then for all $k\ge 0$ we have
\beq\label{lmcg}
C_k(\varphi,u)=\delta_{k,0}\R,
\eeq
where $\delta_{k,h}$ is the Kronecker symbol (see \cite[Example 6.45(a)]{MMP}). Critical groups describe the homology of sublevel sets:

\begin{proposition}\label{subs}
{\rm \cite[Lemma 6.55]{MMP}} Let $\varphi\in C^1(X)$ satisfy \eqref{c}, $a<c<b$ be real numbers s.t.\ $c$ is the only critical value of $\varphi$ in $[a,b]$ and $K_c(\varphi)$ is a finite set. Then for all $k\in\N$
\[H_k(\overline\varphi^b,\overline\varphi^a)=\bigoplus_{u\in K_c(\varphi)}C_k(\varphi,u).\]
\end{proposition}

\noindent
Now assume that
\[\inf_{u\in K(\varphi)}\varphi(u)=:\bar c>-\infty.\]
Then we can as well define the $k$-{\em th critical group of $\varphi$ at infinity} as
\beq\label{cginf}
C_k(\varphi,\infty)=H_k(X,\overline\varphi^c),
\eeq
with $c<\bar c$ (this definition also is invariant with respect to $c$). Critical groups at critical points and at infinity are related by the Poincar\'e-Hopf formula (one of the Morse relations):

\begin{theorem}\label{phf}
{\rm \cite[Remark 6.58]{MMP}} Let $\varphi\in C^1(X)$ satisfy \eqref{c}, $a<b$ be real numbers s.t.\ the set
\[K_a^b(\varphi)=\{u\in K(\varphi): \ a\le\varphi(u)\le b\}\]
is finite. Then,
\[\sum_{k=0}^\infty\sum_{u\in K_a^b(\varphi)}(-1)^k{\rm dim}\,(C_k(\varphi,u))=\sum_{k=0}^\infty(-1)^k{\rm dim}\,(C_k(\varphi,\infty)).\]
\end{theorem}

\subsection*{Notation}\label{ss23}

Throughout the paper, $B_r(x)$ will denote the open ball of radius $r>0$ centered at $x\in\R^N$, and $C>0$ will be a constant whose value may change from line to line.

\section{The sublinear case}\label{sec3}

\noindent
In this section we prove the existence of three non-zero solutions of problem \eqref{pro} when $f(x,\cdot)$ is sublinear at infinity, by means of the second deformation theorem and spectral theory. Precisely, we make on the nonlinearity $f$ the following assumtpions:
\begin{itemize}[leftmargin=1cm]
\item[$\h$] $f:\Omega\times\R\to\R$ is a Carath\'eodory mapping, satisfying
\begin{enumroman}
\item\label{h11} $|f(x,t)|\le a_0(1+|t|^{p-1})$ for a.e. $x\in\Omega$ and all $t\in\R$ ($a_0>0$, $p\in(2,2^*_s)$);
\item\label{h12} $f(x,t)t\ge 0$ for a.e. $x\in\Omega$ and all $t\in\R$;
\item\label{h13} $\displaystyle\limsup_{|t|\to\infty}\frac{F(x,t)}{t^2}\leq 0$ uniformly for a.e. $x\in\Omega$;
\item\label{h14} $\displaystyle\liminf_{t\to 0}\frac{F(x,t)}{t^2}\geq\beta$ uniformly for a.e. $x\in\Omega$ ($\beta>0$).
\end{enumroman}
\end{itemize}

\begin{example}
Let $a\in L^\infty(\Omega)$ be a function s.t. $a(x)\ge\beta>0$ for a.e. $x\in\Omega$, and set for all $(x,t)\in\Omega\times\R$
\[f(x,t)=a(x)\,{\rm sign}(t)\ln(1+|t|).\]
Then, $f$ satisfies hypotheses $\h$.
\end{example}

\noindent
Clearly, by hypothesis $\h$ \ref{h12} problem \eqref{pro} always has the zero solution. First we prove that, for $\beta>0$ big enough, \eqref{pro} has two constant sign solutions:

\begin{proposition}\label{sub2}
Let $\h$ hold with $\beta>\frac{\lambda_s^1(\Omega)}{2}$. Then \eqref{pro} admits at least two non-zero solutions $u_\pm\in\pm{\rm int}\,(C^0_\delta(\overline\Omega)_+)$.
\end{proposition}
\begin{proof}
We define $\varphi$ as in \eqref{phi}. Besides we introduce two truncated energy functionals by setting for all $u\in H^s_0(\Omega)$
\beq\label{phipm}
\varphi_\pm(u)=\frac{\|u\|^2}{2}-\int_\Omega F_\pm(x,u)\,dx,
\eeq
where for all $(x,t)\in\Omega\times\R$ we have set
\[f_\pm(x,t)=f(x,t^\pm), \ F_\pm(x,t)=\int_0^t f_\pm(x,\tau)\,d\tau.\]
We focus on the functional $\varphi_+$. Clearly $\varphi_+\in C^1(H^s_0(\Omega))$. We prove now that $\varphi_+$ is coercive in $H^s_0(\Omega)$, i.e.,
\beq\label{coer}
\lim_{\|u\|\to\infty}\varphi_+(u)=\infty.
\eeq
Indeed, by hypotheses $\h$ \ref{h11}, \ref{h13}, for all $\eps>0$ we can find $C_\eps>0$ s.t.\ for a.e. $x\in\Omega$ and all $t\in\R$
\beq\label{sub21}
0\le F_+(x,t)\le C_\eps+\eps t^2.
\eeq
By Proposition \ref{ev} \ref{ev1} and \eqref{sub21}, we have for all $u\in\s$
\begin{align*}
\varphi_+(u) &\ge \frac{\|u\|^2}{2}-\int_\Omega(C_\eps+\eps u^2)\,dx \\
&\ge \Big(\frac{1}{2}-\frac{\eps}{\lambda^1_s(\Omega)}\Big)\|u\|^2-C_\eps|\Omega|.
\end{align*}
If we choose $\eps<\frac{\lambda^1_s(\Omega)}{2}$, the latter tends to $\infty$ as $\|u\|\to\infty$, so \eqref{coer} follows. Moreover, $\varphi_+$ is sequentially weakly lower semi-continuous in $\s$. Indeed, let $u_n\rightharpoonup u$ in $\s$. Passing if necessary to a subsequence, we may assume $u_n\to u$ in $L^p(\Omega)$ and $u_n(x)\to u(x)$ for a.e. $x\in\Omega$, moreover there exists $g\in L^p(\Omega)$ s.t. $|u_n(x)|\le g(x)$ for a.e. $x\in\Omega$ and all $n\in\N$ (see \cite[Theorem 4.9]{B}), hence
\[\lim_n\int_\Omega F_+(x,u_n)\,dx=\int_\Omega F_+(x,u)\,dx.\]
Besides, by convexity we have
\[\liminf_n\frac{\|u_n\|^2}{2}\ge\frac{\|u\|^2}{2},\]
so
\[\liminf_n\varphi_+(u_n)\ge\varphi_+(u).\]
Thus, there exists $u_+\in\s$ s.t.
\beq\label{sub22}
\varphi_+(u_+)=\inf_{u\in\s}\varphi_+(u).
\eeq
In particular, $u_+\in K(\varphi_+)$. By $\h$ \ref{h12} and Proposition \ref{wmp} we have $u_+\in\s_+$. It remains to prove that $u_+\neq 0$. Here we use our assumption on $\beta$: let $\beta'\in (0,\beta)$ be s.t. $\beta'>\frac{\lambda_s^1(\Omega)}{2}$. By $\h$ \ref{h14}, we can find $\sigma>0$ s.t. $F_+(x,t)>\beta' t^2$ for a.e. $x\in\Omega$ and all $|t|\le\sigma$. Let $\hat u_1\in{\rm int}\,(C^0_\delta(\overline\Omega)_+)$ be defined as in Proposition \ref{ev} \ref{ev1}, then for $\mu>0$ small enough we have $\|\mu\hat u_1\|_\infty\le\sigma$, hence
\begin{align*}
\varphi_+(\mu\hat u_1) &\le \frac{\|\mu\hat u_1\|^2}{2}-\int_\Omega\beta'(\mu\hat u_1)^2\,dx \\
&= \mu^2\Big(\frac{1}{2}-\frac{\beta'}{\lambda_s^1(\Omega)}\big)\|\hat u_1\|^2<0.
\end{align*}
By \eqref{sub22} we have $\varphi_+(u_+)<0$, hence $u_+\neq 0$. By Proposition \ref{hl} we deduce $u_+\in{\rm int}\,(C^0_\delta(\overline\Omega)_+)$. Noting that $\varphi(u)=\varphi_+(u)$ for all $u\ge 0$, we see that $u_+$ is a local minimizer of $\varphi$ in $C^0_\delta(\overline\Omega)$, hence by Proposition \ref{lm} a local minimizer of  $\varphi$ in $\s$. In particular $u_+\in K(\varphi)$, hence $u_+$ is a positive solution of \eqref{pro}.
\vskip2pt
\noindent
Similarly, we find another local minimizer $u_-\in -{\rm int}\,(C^0_\delta(\overline\Omega)_+)$ of $\varphi$, which turns out to be a negative solution of \eqref{pro}.
\end{proof}

\noindent
Now, taking $\beta>0$ even bigger, we achieve a third non-zero solution:

\begin{theorem}\label{sub3}
Let $\h$ hold with $\beta>\frac{\lambda_s^2(\Omega)}{2}$. Then \eqref{pro} admits at least three non-zero solutions $u_\pm\in\pm{\rm int}\,(C^0_\delta(\overline\Omega)_+)$, $\tilde u\in C^0_\delta(\overline\Omega)\setminus\{0\}$.
\end{theorem}
\begin{proof}
First we note, arguing as in the proof of \eqref{coer}, that
\beq\label{coer1}
\lim_{\|u\|\to\infty}\varphi(u)=\infty.
\eeq
Now we prove that $\varphi$ satisfies \eqref{c} (which in this case is equivalent to the Palais-Smale condition). Let $(u_n)$ be a sequence in $\s$, s.t. $|\varphi(u_n)|\le C$ for all $n\in\N$ and $(1+\|u_n\|)\varphi'(u_n)\to 0$ in $H^{-s}(\Omega)$. By \eqref{coer1}, $(u_n)$ is bounded in $\s$. Hence, passing if necessary to a subsequence, we may assume $u_n\rightharpoonup u$ in $\s$, $u_n\to u$ in $L^p(\Omega)$ and $L^1(\Omega)$, and $u_n(x)\to u(x)$ for a.e. $x\in\Omega$, with some $u\in\s$. Moreover, by \cite[Theorem 4.9]{B} there exists $g\in L^p(\Omega)$ s.t. $|u_n(x)|\le g(x)$ for all $n\in\N$ and a.e. $x\in\Omega$. Using such relations along with $\h$ \ref{h11}, we have for all $n\in\N$
\begin{align*}
\|u_n-u\|^2 &= \langle u_n,u_n-u\rangle-\langle u,u_n-u\rangle \\
&= \varphi'(u_n)(u_n-u)+\int_\Omega f(x,u_n)(u_n-u)\,dx-\langle u,u_n-u\rangle \\
&\le \|\varphi'(u_n)\|_*\|u_n-u\|+\int_\Omega a_0(1+|u_n|^{p-1})|u_n-u|\,dx-\langle u,u_n-u\rangle \\
&\le \|\varphi'(u_n)\|_*\|u_n-u\|+a_0(\|u_n-u\|_1+\|u_n\|_p^{p-1}\|u_n-u\|_p)-\langle u,u_n-u\rangle,
\end{align*}
and the latter tends to $0$ as $n\to\infty$. Thus, $u_n\to u$ in $\s$.
\vskip2pt
\noindent
By $\h$ \ref{h12} we have $0\in K(\varphi)$, while from Proposition \ref{sub2} we know that $u_\pm\in K(\varphi)\setminus\{0\}$. We aim at proving existence of a further critical point $\tilde u\in\s$. We argue by contradiction, assuming
\beq\label{sub31}
K(\varphi)=\{0,u_+,u_-\}.
\eeq
It is not restrictive to assume that $\varphi(u_+)\ge\varphi(u_-)$ and that $u_+$ is a strict local minimizer of $\varphi$, so we can find $r\in(0,\|u_+-u_-\|)$ s.t. $\varphi(u)>\varphi(u_+)$ for all $u\in\s$, $0<\|u-u_+\|\le r$. Moreover, we have
\beq\label{sub32}
\eta_r:=\inf_{\|u-u_+\|=r}\varphi(u)>\varphi(u_+).
\eeq
Otherwise, we could find a sequence $(u_n)$ in $\s$ s.t. $\|u_n-u_+\|=r$ for all $n\in\N$, $\varphi(u_n)\to\varphi(u_+)$ and $\varphi'(u_n)\to 0$ in $H^{-s}(\Omega)$ (see \cite[Corollary 5.12]{MMP}). Then, by \eqref{c} we would have $u_n\to\bar u$ in $\s$ for some $\bar u\in\s$, $\|\bar u-u_+\|=r$, hence in turn $\varphi(\bar u)=\varphi(u_+)$, a contradiction.
\vskip2pt
\noindent
Now set
\[\Gamma=\{\gamma\in C([0,1],\s):\,\gamma(0)=u_+,\,\gamma(1)=u_-\},\]
\[c=\inf_{\gamma\in\Gamma}\max_{t\in[0,1]}\varphi(\gamma(t)).\]
By Theorem \ref{mpt} we have $c\ge\eta_r$ and there exists $\tilde u\in K_c(\varphi)$. By \eqref{sub32} we have $\tilde u\neq u_\pm$. So, \eqref{sub31} implies $\tilde u=0$, hence $c=0$. To reach a contradiction, we will construct a path $\gamma\in\Gamma$ s.t.
\beq\label{sub33}
\max_{t\in[0,1]}\varphi(\gamma(t))<0,
\eeq
so that $c<0$. Let $\beta'\in(0,\beta)$, $\theta>0$ be s.t.
\beq\label{sub34}
\beta'>\frac{\lambda_s^2(\Omega)+\theta}{2}.
\eeq
By $\h$ \ref{h14} there exists $\sigma>0$ s.t. $F(x,t)>\beta't^2$ for a.e. $x\in\Omega$ and all $|t|\le\sigma$. Besides, by Proposition \ref{ev} \ref{ev2} there exists $\gamma_1\in\Gamma_1$ s.t.
\beq\label{sub35}
\max_{t\in[0,1]}\|\gamma_1(t)\|^2<\lambda_s^2(\Omega)+\theta.
\eeq
Since $C^\infty_0(\Omega)$ is dense in $\s$ (see \cite[Theorem 2]{FSV}), we can choose $\gamma_1(t)\in L^\infty(\Omega)$ for all $t\in[0,1]$ and $\gamma_1$ continuous with respect to the $L^\infty(\Omega)$-topology. So, by choosing $\eps>0$ small enough, we have $\|\eps\gamma_1(t)\|_\infty\le\sigma$ for all $t\in[0,1]$. Thus, by \eqref{sub35} and recalling that $\|\gamma_1(t)\|_2=1$, we have for all $t\in[0,1]$
\begin{align*}
\varphi(\eps\gamma_1(t)) &\le \frac{\eps^2\|\gamma_1(t)\|^2}{2}-\beta'\eps^2\|\gamma_1(t)\|_2^2 \\
&< \eps^2\Big(\frac{\lambda_s^2(\Omega)+\theta}{2}-\beta'\Big),
\end{align*}
and the latter is negative by \eqref{sub34}. Then, $\eps\gamma_1$ is a continuous path joining $\eps\hat u_1$ and $-\eps\hat u_1$ s.t.
\beq\label{sub36}
\max_{t\in[0,1]}\varphi(\eps\gamma_1(t))<0.
\eeq
By $\h$ \ref{h12} and Proposition \ref{wmp}, it is easily seen that $K(\varphi_+)\subseteq K(\varphi)$. More precisely, by \eqref{sub31}, we have $K(\varphi_+)=\{0,u_+\}$. Set $a=\varphi_+(u_+)$, $b=0$, then $\varphi_+$ satisfies all assumptions of Theorem \ref{sdt}, so there exists a continuous deformation $h_+:[0,1]\times(\overline\varphi^0_+\setminus\{0\})\to(\overline\varphi^0_+\setminus\{0\})$ s.t.
\[\begin{cases}
\text{$h_+(t,u_+)=u_+$ for all $t\in[0,1]$} \\
\text{$h_+(1,u)=u_+$ for all $u\in(\overline\varphi^0_+\setminus\{0\})$} \\
\text{$t\mapsto\varphi_+(h_+(t,u))$ is decreasing for all $u\in(\overline\varphi^0_+\setminus\{0\})$.}
\end{cases}\]
Set for all $t\in[0,1]$
\[\gamma_+(t)=h_+(t,\eps\hat u_1),\]
then $\gamma_+\in C([0,1],\s)$ is a path joining $\eps\hat u_1$ and $u_+$, s.t. $\varphi_+(\gamma_+(t))<0$ for all $t\in[0,1]$. Noting that $\varphi(u)\le\varphi_+(u)$ for all $u\in\s$, we have
\beq\label{sub37}
\max_{t\in[0,1]}\varphi(\gamma_+(t))<0.
\eeq
Similarly, we construct a path $\gamma_-\in C([0,1],\s)$ joining $-\eps\hat u_1$ and $u_-$, s.t.
\beq\label{sub38}
\max_{t\in[0,1]}\varphi(\gamma_-(t))<0.
\eeq
Concatenating $\gamma_+$, $\eps\gamma_1$, and $\gamma_-$ (with convenient changes of parameter) and considering \eqref{sub36} - \eqref{sub38}, we construct a path $\gamma\in\Gamma$ satisfying \eqref{sub33}, against \eqref{sub31} and the definition of the mountain pass level $c$.
\vskip2pt
\noindent
So, we conclude that there exists a fourth critical point $\tilde u\in K(\varphi)\setminus\{0,u_+,u_-\}$, which turns out to be a non-zero solution of \eqref{pro}, concluding the proof.
\end{proof}

\section{The superlinear case}\label{sec4}

\noindent
In this section we prove the existence of three non-zero solutions of problem \eqref{pro} when $f(x,\cdot)$ is superlinear at infinity. Following an idea first appeared in \cite{W}, we will apply the mountain pass theorem and Morse theory. Precisely, we make on the nonlinearity $f$ the following assumtpions:
\begin{itemize}[leftmargin=1cm]
\item[$\hh$] $f:\Omega\times\R\to\R$ is a Carath\'eodory mapping, satisfying
\begin{enumroman}
\item\label{h21} $|f(x,t)|\le a_0(1+|t|^{p-1})$ for a.e. $x\in\Omega$ and all $t\in\R$ ($a_0>0$, $p\in(2,2^*_s)$);
\item\label{h22} $f(x,t)t\le 0$ for a.e. $x\in\Omega$ and all $t\in[-\sigma,\sigma]$ ($\sigma>0$);
\item\label{h23} $f(x,t)t\ge -c_0t^2$ for a.e. $x\in\Omega$ and all $t\in\R$ ($c_0>0$);
\item\label{h24} $\displaystyle\lim_{|t|\to\infty}\frac{F(x,t)}{t^2}=\infty$ uniformly for a.e. $x\in\Omega$;
\item\label{h25} $\displaystyle\liminf_{|t|\to\infty}\frac{f(x,t)t-2F(x,t)}{|t|^q}>0$ uniformly for a.e. $x\in\Omega$ $\displaystyle\Big(q\in\Big(\frac{(p-2)N}{2s},2^*_s\Big)\Big)$.
\end{enumroman}
\end{itemize}
Condition $\hh$ \ref{h25} is a mild version of the classical Ambrosetti-Rabinowitz condition (see \cite{R}), and since $p<2^*_s$ we can always assume $q<p$ in it.

\begin{example}
Let $a,b\in L^\infty(\Omega)$ be s.t. $a(x)\ge\alpha$, $b(x)\ge\beta$ for a.e. $x\in\Omega$ ($\alpha,\beta>0$), and set for all $(x,t)\in\Omega\times\R$
\[f(x,t)=-a(x)t+b(x)|t|^{p-2}t.\]
Then, $f$ satisfies hypotheses $\hh$ with convenient $a_0$, $c_0$, and $\sigma$. This choice of $f$ belongs in the class of concave-convex nonlinearities, whose study (in the classical case $s=1$) started with \cite{ABC}.
\end{example}

\noindent
By hypothesis $\hh$ \ref{h22}, problem \eqref{pro} admits the zero solution. We focus now on constant sign solutions:

\begin{proposition}\label{sup2}
Let $\hh$ hold. Then \eqref{pro} admits at least two non-zero solutions $u_\pm\in\pm{\rm int}\,(C^0_\delta(\overline\Omega)_+)$.
\end{proposition}
\begin{proof}
We define $\varphi$, $\varphi_\pm$ as in \eqref{phi}, \eqref{phipm}. We focus mainly on $\varphi_+$.
\vskip2pt
\noindent
First we prove that $\varphi_+$ satisfies \eqref{c}. Let $(u_n)$ be a sequence in $\s$ s.t. $|\varphi_+(u_n)|\le C$ for all $n\in\N$ and $(1+\|u_n\|)\varphi_+'(u_n)\to 0$ in $H^{-s}(\Omega)$. Then we have for all $n\in\N$
\[-\|u_n\|^2+\int_\Omega f_+(x,u_n)u_n\,dx\le C,\]
\[\|u_n\|^2-2\int_\Omega F_+(x,u_n)\,dx\le C,\]
which imply
\beq\label{sup21}
\int_\Omega\big(f_+(x,u_n)u_n-2F_+(x,u_n)\big)\,dx\le C
\eeq
Clearly $\hh$ \ref{h25} yields
\[\lim_{t\to\infty}\frac{f_+(x,t)t-2F_+(x,t)}{t^q}>0\]
uniformly for a.e. $x\in\Omega$. So we can find $\beta,M>0$ s.t. $f_+(x,t)t-2F_+(x,t)\ge\beta t^q$ for a.e. $x\in\Omega$ and all $t>M$. We claim that $(u_n)$ is bounded in $L^q(\Omega)$. Indeed, for all $n\in\N$ we have
\[\|u_n\|_q^q=\|u_n^+\|_q^q+\|u_n^-\|_q^q.\]
By the previous inequality we have
\begin{align*}
\beta\|u_n^+\|_q^q &= \int_{\{0<u_n\le M\}}\beta u_n^q\,dx+\int_{\{u_n>M\}}\beta u_n^q\,dx \\
&\le \beta M^q|\Omega|+\int_{\{u_n>M\}}\big(f_+(x,u_n)u_n-2F_+(x,u_n)\big)\,dx \\
&\le C+\int_\Omega\big(f_+(x,u_n)u_n-2F_+(x,u_n)\big)\,dx,
\end{align*}
and the latter is bounded by \eqref{sup21}. Besides, using \eqref{np}, we easily have
\[\|u_n^-\|^2\le -\varphi'_+(u_n)(u_n^-)\le \|\varphi'_+(u_n)\|_*\|u_n^-\|,\]
and the latter tends to $0$ as $n\to\infty$. By the continuous embedding $\s\hookrightarrow L^q(\Omega)$, this yields $\|u_n^-\|_q\to 0$ as $n\to\infty$. So we deduce that $\|u_n\|_q$ is bounded in $\R$.
\vskip2pt
\noindent
Using this fact, we want to show that $(u_n)$ is bounded in $\s$ as well. Since $q<p<2^*_s$ in our assumptions, we can find $\tau\in(0,1)$ s.t.
\[\frac{1}{p}=\frac{1-\tau}{q}+\frac{\tau}{2^*_s}.\]
By the interpolation inequality (see \cite[p.\ 93]{B}) and the continuous embedding $\s\hookrightarrow L^{2^*_s}(\Omega)$, we have for all $n\in\N$
\beq\label{sup22}
\|u_n\|_p\le\|u_n\|_q^{1-\tau}\|u_n\|_{2^*_s}^\tau \le C\|u_n\|^\tau.
\eeq
Again by $(1+\|u_n\|)\varphi_+'(u_n)\to 0$ in $H^{-s}(u_n)$ and $\hh$ \ref{h21} we have for all $n\in\N$
\begin{align*}
\|u_n\|^2 &\le \int_\Omega f_+(x,u_n)u_n\,dx \\
&\le \int_\Omega a_0(1+|u_n|^{p-1})|u_n|\,dx \\
&\le C(1+\|u_n\|_1+\|u_n\|_p^p).
\end{align*}
By \eqref{sup22} and the continuous embeddings $\s\hookrightarrow L^1(\Omega),\,L^p(\Omega)$ we see that
\[\|u_n\|^2\le C(1+\|u_n\|+\|u_n\|^{p\tau}).\]
Since $p\tau<2$ we deduce that $(u_n)$ is bounded in $\s$. Now we conclude as in the proof of Theorem \ref{sub3}.
\vskip2pt
\noindent
Now we prove that $\varphi_+$ is unbounded from below. Indeed, let $\hat u_1$ be defined as in Proposition \ref{ev} \ref{ev1}, and recall that $\|\hat u_1\|^2=\lambda_s^1(\Omega)$, $\|\hat u_1\|_2^2=1$. By $\hh$ \ref{h24}, given $\theta>\frac{\lambda_s^1(\Omega)}{2}$ we can find $K>0$ s.t. $F(x,t)\ge\theta t^2$ for a.e. $x\in\Omega$ and all $|t|>M$. For all $\mu>0$ we have
\begin{align*}
\varphi_+(\mu\hat u_1) &= \frac{\mu^2\|\hat u_1\|^2}{2}-\int_{\{\mu\hat u_1\le M\}}F_+(x,\mu\hat u_1)\,dx-\int_{\{\mu\hat u_1>M\}}F_+(x,\mu\hat u_1)\,dx \\
&\le \frac{\mu^2\lambda^1_s(\Omega)}{2}-\int_{\{\mu\hat u_1>M\}}\theta\mu^2\hat u_1^2\,dx+C \\
&\le \mu^2\Big(\frac{\lambda_s^1(\Omega)}{2}-\theta\Big)+\theta M^2|\Omega|+C,
\end{align*}
and the latter goes to $-\infty$ as $\mu\to\infty$. So
\beq\label{sup23}
\lim_{\mu\to\infty}\varphi_+(\mu\hat u_1)=-\infty.
\eeq
We claim that $0$ is a local minimizer for $\varphi_+$. By $\hh$ \ref{h22} we have $F_+(x,t)\le 0$ for a.e. $x\in\Omega$ and all $|t|\le\sigma$. For all $u\in C^0_\delta(\overline\Omega)$ with
\[\|u\|_{0,\delta}\le\frac{\sigma}{{\rm diam}(\Omega)^s},\]
we have $\|u\|_\infty\le\sigma$, hence
\[\varphi_+(u) \ge \frac{\|u\|^2}{2}\ge 0.\]
So, $0$ is a local minimizer of $\varphi_+$ in $C^0_\delta(\overline\Omega)$. By Proposition \ref{lm}, $0$ is as well a local minimizer of $\varphi_+$ in $\s$. As usual, it is not restrictive to assume that $0$ is a strict local minimizer for both $\varphi_+$ and (reasoning as in the proof of \eqref{sub32}) there exists $r>0$ s.t.
\beq\label{sup24}
\eta^+_r:=\inf_{\|u\|=r}\varphi_+(u)>0.
\eeq
By \eqref{sup23} we can find $\mu>0$ s.t. $\|\mu\hat u_1\|>r$ and $\varphi_+(\mu\hat u_1)<0$. Set
\[\Gamma_+=\{\gamma\in C([0,1],\s):\,\gamma(0)=0,\,\gamma(1)=\mu\hat u_1\},\]
\[c_+=\inf_{\gamma\in\Gamma_+}\max_{t\in[0,1]}\varphi_+(\gamma(t)).\]
By Theorem \ref{mpt} we have $c_+\ge\eta_r^+$ and there exists $u_+\in K_{c_+}(\varphi_+)$. From \eqref{sup24} we see that $c_+>0$, hence $u_+\neq 0$. Testing $\varphi_+'(u_+)=0$ with $(u_+)^-\in\s$ and using \eqref{np}, we get
\[-\|(u_+)^-\|^2\ge\varphi'_+(u_+)((u_+)^-)=0,\]
i.e., $u_+\in\s_+$ (note that Proposition \ref{wmp} does not apply here). By $\hh$ \ref{h23} we can apply Proposition \ref{hl} and deduce $u_+\in{\rm int}\,(C^0_\delta(\overline\Omega)_+)$, in particular $u_+\in K(\varphi)$. So, $u_+$ is a positive solution of \eqref{pro}.
\vskip2pt
\noindent
A similar argument, applied to $\varphi_-$, leads to the existence of a negative solution $u_-\in -{\rm int}\,(C_\delta^0(\overline\Omega)_+)$ of \eqref{pro}.
\end{proof}

\noindent
Using the critical groups, we can improve the conclusion of Proposition \ref{sup2} under the same assumptions:

\begin{theorem}\label{sup3}
Let $\hh$ hold. Then \eqref{pro} admits at least three non-zero solutions $u_\pm\in\pm{\rm int}\,(C^0_\delta(\overline\Omega)_+)$, $\tilde u\in C^0_\delta(\overline\Omega)\setminus\{0\}$.
\end{theorem}
\begin{proof}
Reasoning as in the proof of Proposition \ref{sup2} we see that $\varphi$, $\varphi_\pm$ satisfy \eqref{c}, are unbounded from below and have a strict local minimum at $0$. Moreover we know that $0,u_\pm\in K(\varphi)$. We aim at finding a further critical point for $\varphi$. We argue by contradiction, assuming
\beq\label{abs}
K(\varphi)=\{0,u_+,u_-\}.
\eeq
In particular, all critical points of $\varphi$ are isolated. Taking $a<b$ in $\R$ s.t.\ all critical levels of $\varphi$ lie in $(a,b)$, from Theorem \ref{phf} we have
\beq\label{ph}
\sum_{k=0}^\infty(-1)^k\big({\rm dim}\,C_k(\varphi,0)+{\rm dim}\,C_k(\varphi,u_+)+{\rm dim}\,C_k(\varphi,u_-)\big)=\sum_{k=0}^\infty(-1)^k{\rm dim}\,C_k(\varphi,\infty).
\eeq
Now we will compute all critical groups of $\varphi$ both at its critical points and at infinity, then we will plug results into \eqref{ph} to get a contradiction. In doing so, we will also need to compute some critical groups of $\varphi_\pm$.
\vskip2pt
\noindent
We begin with critical groups at infinity: for all integer $k\ge 0$ we have
\beq\label{ginf}
C_k(\varphi,\infty)=C_k(\varphi_\pm,\infty)=0.
\eeq
We focus on $\varphi$ (the argument for $\varphi_\pm$ is analogous). We recall from the proof of Proposition \ref{sup2} that
\[\min\{\varphi(u_+),\varphi(u_-)\}>\varphi(0)=0.\]
We denote the unit sphere in $\s$ by
\[S=\{u\in\s:\,\|u\|=1\}.\]
Reasoning as in the proof of \eqref{sup23} we see that for all $u\in S$
\beq\label{sup31}
\lim_{\mu\to\infty}\varphi(\mu u)=-\infty.
\eeq
Moreover, taking $c<0$ small enough, we have for all $v\in\varphi^{-1}(c)$
\beq\label{sup32}
\varphi'(v)(v)<0.
\eeq
Indeed, by $\hh$ \ref{h25} there exists $\beta,M>0$ s.t. $f(x,t)t-2F(x,t)\ge\beta |t|^q$ for a.e. $x\in\Omega$ and all $|t|>M$. Then, using also $\hh$ \ref{h21}, for all $v\in\varphi^{-1}(c)$ we have
\begin{align*}
\varphi'(v)(v) &= \|v\|^2-\int_\Omega f(x,v)v\,dx \\
&= 2\varphi(v)-\int_\Omega\big(f(x,v)v-2F(x,v)\big)\,dx \\
&\le 2c-\int_{\{|v|>M\}}\beta|v|^q\,dx+\int_{\{|v|\le M\}}\Big(a_0(|v|+|v|^p)+a_0\Big(|v|+\frac{|v|^p}{p}\Big)\Big)\,dx \\
&\le 2c-\beta\|v\|_q^q+\beta M^q|\Omega|+C(M+M^p)|\Omega| \\
&\le 2c+C_M,
\end{align*}
with a constant $C_M>0$ only depending on $M$. So, choosing
\[c<\min\Big\{-\frac{C_M}{2},\,\inf_{\|u\|\le 1}\varphi(u)\Big\},\]
we get \eqref{sup32}. Now we apply the implicit function theorem \cite[Theorem 7.3]{MMP} to the function $(\mu,u)\mapsto\varphi(\mu u)$ defined in $(1,\infty)\times S$. By \eqref{sup32} we have for all $(\mu,u)\in(1,\infty)\times S$ with $\varphi(\mu u)=c$
\[\frac{\partial}{\partial\mu}\varphi(\mu u)=\frac{\varphi'(\mu u)(\mu u)}{\mu}<0,\]
hence there exists a continuous mapping $\rho:S\to(1,\infty)$ s.t. for all $(\mu,u)\in(1,\infty)\times S$
\[\varphi(\mu u) \ \begin{cases}
>c & \text{if $\mu<\rho(u)$} \\
=c & \text{if $\mu=\rho(u)$} \\
<c & \text{if $\mu>\rho(u)$.}
\end{cases}\]
So we have
\[\overline\varphi^c=\{\mu u:\,u\in S,\,\mu\in[\rho(u),\infty)\}.\]
Set also
\[E=\{\mu u:\,u\in S,\,\mu\ge 1\}.\]
We can define a continuous deformation $h:[0,1]\times E\to E$ by setting for all $(t,\mu u)\in[0,1]\times E$
\[h(t,\mu u)=\begin{cases}
(1-t)\mu u+t\rho(u)u & \text{if $\mu<\rho(u)$} \\
\mu u & \text{if $\mu\ge\rho(u)$,}
\end{cases}\]
so $\overline\varphi^c$ is a strong deformation retract of $E$. Besides, we define another continuous deformation $\tilde h:[0,1]\times E\to E$ by setting for all $(t,\mu u)\in[0,1]\times E$
\[\tilde h(t,\mu u)=(1-t)\mu u+tu,\]
showing that $S$ is also a strong deformation retract of $E$. By the choice of $c$, \eqref{cginf}, and \cite[Corollary 6.15]{MMP} we have for all $k\ge 0$
\[C_k(\varphi,\infty)=H_k(\s,\overline\varphi^c)=H_k(\s,E)=H_k(\s,S),\]
and the latter is $0$ by \cite[Propositions 6.24, 6.25]{MMP} (recall that $S$ is contractible in itself, as ${\rm dim}\,\s=\infty$). Thus we have \eqref{ginf}.
\vskip2pt
\noindent
We compute now the critical points at $0$: for all $k\ge 0$
\beq\label{g0}
C_k(\varphi,0)=C_k(\varphi_\pm,0)=\delta_{k,0}\R.
\eeq
Reasoning as in the proof of Proposition \ref{sup2} and using \eqref{abs}, we see that $0$ is a strict local minimizer of $\varphi$, so \eqref{g0} follows from \eqref{lmcg} (the argument for $\varphi_\pm$ is analogous).
\vskip2pt
\noindent
Finally we compute the critical groups at $u_\pm$: for all $k\ge 0$ we have
\beq\label{g+}
C_k(\varphi,u_\pm)=\delta_{k,1}\R.
\eeq
We consider $u_+$ (the argument for $u_-$ is analogous). First we note that
\beq\label{hom}
C_k(\varphi,u_+)=C_k(\varphi_+,u_+).
\eeq
Indeed, for all $\tau\in[0,1]$ we define $\psi_\tau\in C^1(\s)$ by setting for all $u\in\s$
\[\psi_\tau(u)=(1-\tau)\varphi(u)+\tau\varphi_+(u).\]
Clearly we have $u_+\in K(\psi_\tau)$ for all $\tau\in[0,1]$. Moreover, $u_+$ is an isolated critical point of $\psi_\tau$ uniformly with respect to $\tau$, as we shall prove arguing by contradiction: assume that there exist sequences $(u_n)$ in $\s\setminus\{u_+\}$, $(\tau_n)$ in $(0,1)$ s.t. $u_n\to u_+$ in $\s$, $\tau_n\to 0$, and $\psi'_{\tau_n}(u_n)=0$ in $H^{-s}(\Omega)$ for all $n\in\N$. Then, for all $n\in\N$, $u_n$ is a solution of the \eqref{pro}-type problem
\[\begin{cases}
\fl u_n=(1-\tau_n)f(x,u_n)+\tau_n f_+(x,u_n) & \text{in $\Omega$} \\
u_n=0 & \text{in $\Omega^c$,}
\end{cases}\]
with a reaction term satisfying ${\bf H}_0$ uniformly (i.e., with $a_0$, $p$ independent of $n$). By Proposition \ref{apb} the sequence $(u_n)$ is bounded in $L^\infty(\Omega)$, and by Proposition \ref{reg} there exist $\alpha\in(0,1)$, $C>0$ s.t.\ for all $n\in\N$ we have $u_n\in C^\alpha_\delta(\overline\Omega)$ and $\|u_n\|_{\alpha,\delta}\le C$.
\vskip2pt
\noindent
By the compact embedding $C^\alpha_\delta(\overline\Omega)\hookrightarrow C^0_\delta(\overline\Omega)$, passing if necessary to a subsequence we have $u_n\to u_+$ in $C^0_\delta(\overline\Omega)$, hence $u_n\in{\rm int}\,(C^0_\delta(\overline\Omega)_+)$ for all $n\in\N$ large enough. This in turn implies that $u_n$ is a solution of \eqref{pro}, i.e., a critical point of $\varphi$ different from $0$ and $u_\pm$, against \eqref{abs}.
\vskip2pt
\noindent
So, by homotopy invariance of critical groups (see \cite[Theorem 5.6]{C}), we see that $C_k(\psi_\tau,u_+)$ is independent of $\tau\in[0,1]$. Noting that $\psi_0=\varphi$ and $\psi_1=\varphi_+$, and thus achieve \eqref{hom}.
\vskip2pt
\noindent
By \eqref{hom}, we are reduced to computing $C_k(\varphi_+,u_+)$. Recall that $K(\varphi_+)=\{0,u_+\}$, and fix $a,b\in\R$ s.t.
\[a<\varphi_+(0)<b<\varphi_+(u_+),\]
then set $A=\overline\varphi_+^a$, $B=\overline\varphi_+^b$. We have $A\subset B$, and the following long sequence is exact due to \cite[Proposition 6.14]{MMP}:
\[\ldots\to H_k(\s,A)\xrightarrow{j_*} H_k(\s,B)\xrightarrow{\partial_*} H_{k-1}(B,A)\xrightarrow{i_*} H_{k-1}(\s,A)\to\ldots\]
Here $j_*$, $i_*$ are the group homomorphisms induced by the inclusion mappings $j:(\s,A)\to(\s,B)$ and $i:(B,A)\to(\s,A)$, respectively, and $\partial_*$ is the boundary homomorphism (see \cite[Definition 6.9]{MMP}). By Proposition \ref{subs} and \eqref{cginf} we have
\[H_k(\s,A)=C_k(\varphi_+,\infty), \ H_k(\s,B)=C_k(\varphi_+,u_+), \ H_{k-1}(B,A)=C_{k-1}(\varphi_+,0).\]
So, recalling \eqref{ginf}, the exact sequence rephrases as
\[0\to C_k(\varphi_+,u_+)\to C_{k-1}(\varphi_+,0)\to 0,\]
which by \eqref{g0} yields
\[C_k(\varphi_+,u_+)=\delta_{(k-1),0}\R=\delta_{k,1}\R.\]
By \eqref{hom}, we get \eqref{g+}.
\vskip2pt
\noindent
Plugging \eqref{ginf}, \eqref{g0}, and \eqref{g+} into \eqref{ph}, we have
\[\sum_{k=0}^\infty(-1)^k(\delta_{k,0}+2\delta_{k,1})=0,\]
namely $-1=0$, a contradiction. Thus, \eqref{abs} cannot hold, i.e., there exists a further critical point $\tilde u\in K(\varphi)\setminus\{0,u_+,u_-\}$. By Proposition \ref{reg}, we see that $u\in C^0_\delta(\overline\Omega)$ and is a solution of \eqref{pro}.
\end{proof}

\begin{remark}
A comparison between Theorems \ref{sub3} and \ref{sup3} is now in order. Though formally the statements of such results coincide, the underlying structure of the critical set $K(\varphi)$ is widely different in the two cases: in the sublinear case we have two local minimizers $u_+$, $u_-$ and a third non-zero critical point $\tilde u$, typically of mountain pass type; while in the superlinear case we have two mountain pass-type points $u_+$, $u_-$ and a third non-zero critical point of undetermined nature $\tilde u$.
\end{remark}

\noindent
{\small {\bf Aknowledgement.} The second author is a member of the Gruppo Nazionale per l'Analisi Matematica, la Probabilit\`a e le loro Applicazioni (GNAMPA) of the Istituto Nazionale di Alta Matematica (INdAM).}

\end{document}